%% file: nonexistence.tex
\documentclass[12pt,final]{amsart}

\usepackage{showkeys}
\usepackage{xypic}
\include{mydefs}

\theoremstyle{remark}
\newtheorem{example}{Example}[section]
\newtheorem{notation}[example]{Notation}
\newtheorem{definition}[example]{Definition}

\theoremstyle{plain}
\newtheorem{proposition}[example]{Proposition}
\newtheorem{corollary}[example]{Corollary}
\newtheorem{theorem}[example]{Theorem}

\newcommand{\moduli}[1]{\text{M}(\rho, #1)}
\newcommand{\B}{{\mathfrak B}}
\renewcommand{\Q}{{\mathfrak Q}}
\newcommand{\N}{{\mathfrak N}}

\newcommand{\isom}{\text{\underline{Isom}}}
\renewcommand{\shom}{\text{\underline{Hom}}}
\newcommand{\sspec}{\text{\underline{Spec}}}
\newcommand{\sym}{\text{Sym}}

\newcommand{\E}{{\mathcal E}}
\newcommand{\K}{{\mathcal K}}
\newcommand{\F}{{\mathcal F}}
\newcommand{\A}{{\mathcal A}}
\newcommand{\Ba}{{\mathcal B}}

\newcommand{\lie}{\mathrm{Lie}}
\newcommand{\ad}{\mathrm{ad}}
\newcommand{\Sln}{\mathrm{SL}_n}
\newcommand{\Sptn}{\mathrm{Sp}_{2n}}
\newcommand{\PSptn}{\mathrm{PSp}_{2n}}
\newcommand{\GSptn}{\mathrm{GSp}_{2n}}

\newcommand{\Sotn}{\mathrm{SO}_{2n}}
\newcommand{\PSotn}{\mathrm{PSO}_{2n}}

\newcommand{\Gln}{\mathrm{GL}_n}
\newcommand{\rx}{\sqrt{x}}
\newcommand{\diag}{\mathrm{diag}}
\newcommand{\Ext}{\mathrm{Ext}}
\newcommand{\Hom}{\mathrm{Hom}}
\newcommand{\End}{\mathrm{End}}
\newcommand{\orb}{\mathrm{orb}}
\newcommand{\Br}{\mathrm{Br}}

\begin{document}

\title[Nori's Obstruction]{On Nori's Obstruction to Universal Bundles.}

\author{Emre Coskun, Ajneet Dhillon and Nicole Lemire}

\address{Department of Mathematics \\ University of Western Ontario \\ London Ontario N6A 5B7}

\email{ecoskun@uwo.ca}
\email{adhill3@uwo.ca}
\email{nlemire@uwo.ca}

\begin{abstract}Let $G$ be $\Sln, \Sptn$ or $\Sotn$. We consider the moduli space $M$ of 
semistable principal $G$-bundles over a curve $X$. Our main result is that if
$U$ is a Zariski open subset of $M$ then there is no universal bundle on
$U\times X$.
\end{abstract}

\maketitle

\section{Introduction}
We work over an algebraically closed field $k$  of characteristic zero. Let
$X$ be a smooth projective curve over $k$. If $G$ is a semisimple algebraic group, a moduli
space of semistable $G$-bundles over $X$ is constructed in various works, for example
\cite{ramanathan:96a}, \cite{ramanathan:96b} and \cite{schmitt:02}. We outline the construction
below in section \ref{s:moduli}. The construction of the moduli space
depends upon fixing a faithful representation $\rho:G\rightarrow\Sln$. We will denote the obtained
moduli space by
$\moduli{G}$. 

The main result of this paper is that a universal bundle cannot exist over the generic point of $\moduli{G}^0$
 when $G$ is a classical group, not of adjoint type, with its standard representation. (Here $\moduli{G}^0$
denotes the connected component containing the trivial bundle. It is an irreducible algebraic variety.)
When $G$ is of adjoint type it is known that a universal bundle does indeed exist.
Our result is closely related to the work of Balaji, Biswas, Nagaraj and Newstead in
\cite{balaji:06}. They show that for a connected semisimple algebraic group, a universal bundle
exists over $M'$ if and only if $H$ is of adjoint type. Here $M'$ is the open subset
parameterising bundles with automorphism group the centre of $H$. Using \cite[IV Corollary 2.6]{milne:80}
one sees that our results are a proper subset of those in \cite{balaji:06}. The
real interest in the results of this paper comes from the methods. The proof
of the theorem in \cite{balaji:06} requires a detour in the world of
topology while our proof is entirely in the realm of algebra. The reason we
are stuck with characteristic zero arises from the fact that we are applying 
Luna's slice theorem.

The method of proof that we use originates from some ideas of M. Nori. The kernel
of our argument can be found in \cite[Part 6]{seshadri:82}. After completing our work we learned that
N. Hoffman had announced different proofs valid for all reductive groups with non-trivial centres.
This work is in the process of being written.

\textbf{Acknowledgements} The authors would like to thank V. Balaji and C. Seshadri for pointing out the
existence of \cite{seshadri:82}.

\section{Construction of $\moduli{G}$}
\label{s:moduli}
We will follow the construction in \cite{schmitt:02}, however we do not need
all the subtleties of the construction that come from wanting to compactify
a moduli space. For the convenience of the reader, we outline the construction
of an open subset of the moduli space in \cite{schmitt:02} that we will denote by
$\moduli{G}$. 

\subsection{Reductions of structure group}
Recall that we have fixed a faithful representation
$\rho:G\rightarrow\Sln$. If $\E$ is a rank $n$ vector bundle
then the associated $\Gln$-principal bundle is the scheme of isomorphisms
$$
\isom(\struct{X}^n,\E).
$$
We are interested in reductions of structure group of 
this principal bundle to $G$. This is nothing but a section of the
fibration
$$
\isom(\struct{X}^n,\E)/G \rightarrow X.
$$
There is an inclusion
$$
\isom(\struct{X}^n,\E)/G\subseteq \shom(\struct{X}^n,\E)/\!\!/G =\sspec(\sym^*(\struct{X}^n\otimes \E^\vee)^{G}).
$$
Hence a reduction of structure group is just an algebra homomorphism 
$$
\tau:\sym^*(\struct{X}^n\otimes \E^\vee)^G\rightarrow \struct{X}
$$
such that the induced section of
$$
\shom(\struct{X}^n,\E)/\!\!/G \rightarrow X
$$
has image inside of the isomorphism locus.

\subsection{Semistability}
We recall the notions of stability and semistability using \cite{gomez:08}
as our reference.
Let $\gamma:\gm\rightarrow G$ be a one parameter subgroup
of $G$. Together with the representation $\rho:G\rightarrow\text{SL}(W)$
it defines a weight space decomposition
$$
W=\oplus_{n\in \Z} W_n.
$$
We may order the weights $n_1<n_2<\dots <n_k$.  There is an induced flag
$$
F_1\subseteq F_2\subseteq\ldots \subseteq F_k = W
$$
where
$$
F_i = \oplus_{j<i} W_{n_j}.
$$
The parabolic subgroup of $G$ preserving this flag will be denoted by $Q(\gamma)$.
We also set $\alpha_i = \frac{n_{i+1}-n_i}{\dim W}$.

Let $P\rightarrow X$ be a principal $G$-bundle with associated vector bundle
$\rho_*(P)$. A reduction of structure group $\tau$ of $P$ to $Q(\gamma)$ yields
a filtration
$$
0\subseteq \A_1\subset \A_2\dots\subset \A_k=\rho_*(P),
$$
which we denote by $P(\tau,\gamma)_\bullet$.
Finally set 
$$
L(P,\tau,\gamma) = \sum \alpha_i(\text{rk}\A_i\deg(\A) - \text{rk}\A \deg(A_i)).
$$
\begin{definition}
 We say that $P$ is \emph{semistable} if 
$L(P,\tau,\gamma)\ge 0$ for every one parameter subgroup $\gamma$
and reduction of structure group $\tau$.
\end{definition}

Note that since we are over a curve this definition is equivalent to the
definition given in \cite{schmitt:02}. Further, since the $\alpha_i$ are positive,
the trivial principal bundle is semistable.

Given a positive rational number, $\delta$, there is a related notion of 
$\delta$-semistability that we do not recall here as we do not need it. However, let us note the
following result.

\begin{theorem}
 Fix a Hilbert polynomial $F$. There is a positive rational number
$\delta_0$ such that for every $\delta>\delta_0$ and every principal bundle
$P$ with $\rho_*(P)$ having Hilbert polynomial $F$, the following are equivalent:
\begin{enumerate}
 \item $P$ is semistable.
 \item $P$ is $\delta$-semistable
\end{enumerate}
\end{theorem}

\begin{proof}
 See \cite[Theorem 5.4.1]{gomez:08}.
\end{proof}

In \cite{schmitt:02} and \cite{gomez:08}, moduli spaces for $\delta$-semistable bundles
are constructed which depend upon an extra parameter. We will always assume
that the parameter is chosen so the above theorem applies.

\subsection{The Parameter Space}
\label{s:parameter}

The first step towards constructing a moduli space for semistable
bundles is to construct a parameter space for them. Consider the family
$\B$ of vector bundles of rank $n$ and trivial determinant that
admit a reduction of structure group to a semistable $G$-bundle. It is shown
that $\B$ is bounded in \cite[Remark 3.7]{schmitt:02}.

Using boundedness, we can find an integer $N$ so that for all $m\ge N$ and every 
$\F\in\B$ we have
\begin{itemize}
 \item $\coh{i}{X,\F(m)}=0$ for $i>0$
 \item $\F(m)$ is generated by global sections.
\end{itemize}

Set $W=\coh{0}{X,\F(N)}$ for some chosen $\F\in\B$.
We have a quot scheme $\Q$ of quotients of $\struct{X}(-N)\otimes W$ of
degree $0$ and rank $n=\dim \rho$. We have on $\Q\times X$ a universal quotient
$$
W\otimes\pi^*\struct{X}(-N)\rightarrow \F^{\text{univ}},
$$
where $\pi$ is the projection $\Q \times X \to X$. There is an open subset $U\subset \Q$ parameterising locally free quotients.
In \cite[page 1197-1198]{schmitt:02}, a scheme $\N'$ and a map
$$
\N'\rightarrow U
$$
is constructed. The fibre of the scheme $\N'$ over the quotient
$$
W\otimes\struct{X}(-N)\rightarrow \F
$$
parameterises  the collection of algebra homomorphisms 
$$
\tau:\sym^*(\struct{X}^n\otimes \F^\vee)^G\rightarrow \struct{X}.
$$
The object of interest in \cite{schmitt:02} is
$$
 \N' /\!\!/ \text{GL}(W).
$$
There is an open subscheme $\N$ parameterising those algebra homomorphisms
$$
\tau:\sym^*(\struct{X}^n\otimes \F^\vee)^G\rightarrow \struct{X}.
$$
that come from reductions of structure group. We are interested in the quotient
$$
\moduli{G} = \N /\!\!/ \text{GL}(W).
$$

\section{A Little Deformation Theory}

\subsection{Deformations of Principal Bundles}
We begin by recalling some facts from the deformation theory
of principal bundles. Proofs can be found in
\cite{illusie:71}, \cite{illusie:72} and \cite{illusie:71b}.
We do not need the full force of the results in the
cited references as the group scheme $G\times X\rightarrow X$
is smooth.

For a $G$-bundle $P\rightarrow X$ we denote by
$\ad(P)$ the adjoint bundle which is defined by
$$
\ad(P) = P\times_{G,\ad}\lie(G).
$$

Consider an extension of local Artinian $k$-algebras
$$
0\rightarrow I \rightarrow A' \rightarrow A \rightarrow 0
$$
with $I^2 = 0 $.

\begin{theorem}\label{t:ill}
 Let $P_A$ be a $G$-bundle on $X_A$. Then this bundle
extends to $X_{A'}$. The collection of all such lifts 
is a torsor under 
$$
\coh{1}{X_A,\ad(P_A)\otimes I}.
$$
\end{theorem}

\begin{proof}
 See \cite[Theorem 2.6]{illusie:71b}
and \cite[Remark 2.6.1]{illusie:71b}. Note that
$H^2$ vanishes as we are on a curve.
\end{proof}

Next we need some facts about flat deformations of coherent
sheaves.

\begin{theorem}
Fix a sheaf $\E$ on $X$ and consider a flat
family of coherent sheaves $\F_A$ on $X_A$ that
fit into an exact sequence
$$
0\rightarrow \K_A\rightarrow \E\otimes_k A \rightarrow \F_A
\rightarrow 0 
$$
Denote by $\F$ and $\K$ the restriction to the closed point
of $A$. If $$\Ext^1(\K,\F)=0$$ then we may extend
the exact sequence to $X_{A'}$ so that $\F_A$ extends to a flat
family. Furthermore, the collection of all lifts is a torsor under
$$ 
\Hom(\K_A, \F_A\otimes I).
$$
\end{theorem}

\begin{proof}
 See \cite{sernesi:86}.
\end{proof}

\subsection{The Smoothness of $\N$}

\begin{proposition}
 Let $p$ be a point of $\N$ then
$\N$ is smooth at $p$.
\end{proposition}

\begin{proof}
 Let $p$ correspond to the pair $(q:\struct{X}(-N)\otimes W\rightarrow \F,\sigma)$.
Let $n$ be the rank of $\F$ and let $P$ be the bundle corresponding to $\sigma$.
So $P$ is defined by a cartesian diagram
$$
\xymatrix{
\isom(\struct{X}^n,\F) \ar[r] & \isom(\struct{X}^n,\F)/G \\
P \ar[u] \ar[r] & X \ar[u]^\sigma
}
$$
Consider an extension of local Artinian $k$-algebras 
$$
0\rightarrow I \rightarrow A'\rightarrow A \rightarrow 0
$$
with $I^2=0$ and extensions $q_A,\F_A,P_A, \sigma_A$ of the above data to 
$X_{A}$. We need to extend $q_A$ and $P_A$ to $X_{A'}$ in a compatible way. 
Note that $\Ext^1(\K,\F)=0$ since 
$$\Ext^1(\struct{X}(-N)\otimes W, \F)\to \Ext^1(\K,\F)\to \Ext^2(\F,\F)$$
is exact, $\Ext^1(\struct{X}(-N)\otimes W,\F)\cong
H^1(X,\F(N))\otimes W^{\vee}=0$ by the choice of the quot scheme in \ref{s:parameter} and $\Ext^2(\F,\F)=0$ since $X$ is a curve.
Then, by the last proposition, $\F_A$ extends to a flat family.
Let $S$ be the set of extensions of $\F_{A}$  to $X_{A'}$ and $T$ the set of extensions
of $P_A$ to $X_{A'}$. We have a map given by extension of structure group
$$
\psi : T \rightarrow S.
$$ 
We need to choose $q_{A'}$ and $P_{A'}$ so that $\psi(P_{A'})=\F_{A'}$.
The set $S$ is a torsor under $H^1(X_A,\ad(\F)\otimes I)=\Ext^1(\F_A,\F_A\otimes I)$. By the choice of
$N$ in \ref{s:parameter}, we see that $\Ext^1(\struct{X}(-N)\otimes W,\F_A\otimes I)
\cong H^1(X,\F_A(N)\otimes I)\otimes W^{\vee}=0$ so that there is a surjection
$$
e:\text{Hom}(\K_A,\F_A\otimes I)\rightarrow \text{Ext}^1(\F_A,\F_A\otimes I)
$$
arising from the exact sequence
$$
0\rightarrow \K_A\rightarrow \struct{X_A}(-N)\otimes W \rightarrow \F_A
\rightarrow 0. 
$$
 Furthermore, $\text{Hom}(\K_A,\F_A\otimes I)$ classifies
extensions of $q_A$ to $X_{A'}$. One checks if we have two extensions $q_1$ and
$q_2$ then the classes in $S$ of $\F_1$ and $\F_2$ satisfy
$$
e(q_2-q_1) + \F_1 = \F_2\quad\text{in }S, 
$$
where $+$ sign denotes the torsor action described in \ref{t:ill}.
We have
$$
\coh{1}{X_A, \text{ad}({P_A})\otimes I} \subseteq \coh{1}{X_A, \text{ad}({\F_A})\otimes I}
=\text{Ext}^1(\F_A,\F_A\otimes I).
$$
Choose a lift $(q_{A'},\F_{A'})$ of $(q_A,\F_A)$ and a lift $P_{A'}$ of $P_A$. We can hence
find an element $x$ of $\text{Hom}(\K_A,\F_A\otimes I)$ so that
$$e(x)+ \F_{A'} = \psi(P_{A'}).$$
\end{proof}

\subsection{The Slice Theorem}

Consider a quotient of the form
$$
0\rightarrow \K\rightarrow \struct{X}(-N)\otimes W
\rightarrow \struct{X}^n\rightarrow 0
$$
and the canonical section $\sigma:X\rightarrow X \times \Gln/G$.
This pair gives a point $0$ of $\N$ which corresponds to the 
trivial $G$ bundle $G\times X\to X$
and we wish to apply Luna's
\'etale slice theorem to the action of ${\rm GL}(W)$ on $\N$ near
$0$. 

Firstly notice that the surjection
identifies
$$
W\cong\coh{0}{X,\struct{X}^n(N)}.
$$
This in turn gives a direct sum decomposition to $W$
which allows us to identify $\Gln$ with a subgroup of ${\rm GL}(W)$.

The discussion of the preceeding subsection allows us to identify the
Zariski tangent space of $\N$ at $0$ with the following fibred product
in the category of vector spaces

$$
\xymatrix{
 T_0(\N) \ar[r] \ar[d]              & \coh{1}{X,\lie(G)\otimes\struct{X}} \ar[d] \\
\text{Hom}(\K,\struct{X}^n) \ar[r] & \coh{1}{X,\lie(\Gln)\otimes\struct{X}}.
}
$$

Note here that 
$$H^1(X,\ad(G\times X))=H^1(\lie(G)\otimes \struct{X})\cong \lie(G)^g.$$
Using relevant cohomology sequences, we see that the vertical arrows are injective and the horizontal ones surjective.
Note also that 
$$\Hom(\struct{X}(-N)\otimes W,\struct{X}^n)\cong W^{\vee}\otimes W\cong \lie({\rm GL}(W))$$

The differential of the orbit map at $0$ gives a morphism
$$
d\text{orb}_0 : \lie(\text{GL}(W)) \rightarrow T_0(\N).
$$
Composing this with the inclusion into $\text{Hom}(\K,\struct{X}^n)$
and recalling the definition of $W$, 
we may identify it with the natural map 
$$
\text{Hom}(\struct{X}(-N)\otimes W,\struct{X}^n)\rightarrow \text{Hom}(\K,\struct{X}^n).
$$
Putting this all together we have

\begin{proposition} 
\label{p:normal}
Let $0$ be a point of $\N$ corresponding
to the trivial $G$-bundle.
 \begin{enumerate}
  \item The stabiliser of  $0$ inside $\text{GL}(W)$ is $G$.
  \item The normal space to $\orb(0)$ can be identified with
$$
\coh{1}{X,\ad(G\times X)} = \underbrace{\lie(G)\times\ldots\times \lie(G)}_{g\text{ times}}.
$$
  \item The natural action of the stabiliser on the normal bundle can be identified with 
the adjoint action.
 \end{enumerate}
\end{proposition}

\begin{proof}
 The first two statements are straightforward.

It suffices to prove (3) in the case where $G=\Gln$, that is
the full quot scheme. To see it in this case, let us first spell out the 
correspondence between $\text{Hom}(\K,\struct{X}^n)$ and lifts of 
the extension
$$
0\rightarrow \K\rightarrow \struct{X}(-N)\otimes W
\rightarrow \struct{X}^n\rightarrow 0
$$
to the ring of dual numbers. To give a lift of the extension
is the same as giving a subsheaf
$$\tilde{\K}\subseteq \struct{X'}(-N)\otimes W,$$
lifting the subsheaf $\K$. One defines the subsheaf
corresponding to a homomorphism
$$\phi:\K\rightarrow \struct{X}^n$$
by 
$$
\tilde{\K} =\{x+\epsilon y\mid \phi(x) = \bar{y}\}
$$
where $\bar{y}$ is the image inside $\struct{X}^n$.
A group element $g\in \Gln$ produces a diagram
$$
\xymatrix
{
0\ar[r] &\K \ar[r] \ar[d]^g & \struct{X}(-N)\otimes W \ar[r] \ar[d]^g& \struct{X}^n  \ar[r] \ar[d]^g&0\\
0\ar[r] &\K \ar[r] & \struct{X}(-N)\otimes W \ar[r] & \struct{X}^n  \ar[r] &0,
}
$$
by our identification of $\Gln$ with a subgroup of ${\rm GL}(W)$ and the
five lemma. Tracing through the identifications above we find that 
the action on $\text{Hom}(\K,\struct{X}^n)$ is given by
$$
\phi\mapsto g\circ\phi\circ g^{-1}.
$$
This induces the adjoint action on the quotient
$$
\text{Hom}(\K,\struct{X}^n) \rightarrow \coh{1}{X,\lie(\Gln)\otimes\struct{X}}
\rightarrow 0.
$$
\end{proof}

\begin{notation}
 We write 
$$
Z(\lie(G),g) = \lie(G)\times\ldots\times\lie(G).
$$
Frequently we will abuse notation and just write $Z$ or $Z(\lie(G))$
when it is clear from the context what is meant.
\end{notation}

\begin{corollary}
 The completion of the local ring of $\moduli{G}$ at the trivial bundle 
is isomorphic to $\widehat{\struct{Z/\!\!/G,0}}$. 
\end{corollary}

\begin{proof}
 This is just an application of the slice theorem, see
\cite{drezet:04}.
\end{proof}

\section{Geometric Quotients and Azumaya Algebras}
\label{s:azu}

We let $\Lambda=k<x_1,\ldots,x_g>$ be a polynomial ring in $g$
noncommuting variables. If $R$ is a $k$-algebra then an $R$-point
of $Z(\lie(\Gln))$ is equivalent to giving a $R\otimes \Lambda$-module
structure on $R^n$.

Fix a vector space $V$ with a nondegenerate bilinear form $B$ that is 
either symmetric or alternating. The form $B$ determines an anti-automorphism
$r$ of $\End_k(V)$ determined by 
$$
B(v,f(w)) = B(rf(v),w)
$$
for all $f\in\text{End}_k(V)$ and $v,w\in V$. We denote the corresponding 
group of automorphisms of $V$ preserving the form by $G$. To give an 
$R$-point of $Z(\lie(G))$ is the same as giving a  $R\otimes \Lambda$-module
structure on $R\otimes V$ such that if we view $x_i\in\text{End}_R(R\otimes V)$ we have
$r(x_i) = - x_i$.

\subsection{Stability}

In this subsection we will give a  description of  the stable locus of $Z(\lie(G))$ for the adjoint action
of $G$ in terms of the functor of points description above. We denote
the stable locus by $Z(\lie(G))^s$. Fix a representation $V$ of $G$ and $v\in V$
and a one parameter subgroup $\lambda : \gm \rightarrow G$.
We can decompose $v =\sum_i v_{n_i}$ where $v_{n_i}\ne 0$ has weight $n_i$.
One defines
$$
\mu(v,\lambda) = -\min\{n_i\}.
$$

\begin{proposition}
\label{p:simple}
 We have $x\in Z(\lie(\Gln))^s$ if and only if the corresponding $\Lambda$-module
is simple.
\end{proposition}

\begin{proof}
 Suppose that $x=(x_1,x_2,\ldots,x_g)$ is not stable. Then there a 1-parameter subgroup
$\lambda:\gm\rightarrow\Gln$ with $\mu(x,\lambda)\le 0$, see \cite[page 49]{git:94}. We may assume that
$\lambda$ is diagonal or even
$$
\lambda(t) = \left(
\begin{array}{cccc}
 t^{r_1} & & &\\
         & t^{r_2} &  & \\
         &         &  \ddots & \\
         &         &   & t^{r_n}
\end{array}
\right)
$$
with $r_1\ge r_2\ge\ldots\ge r_n$.
Then the fact $\mu(x,\lambda)\le 0$ translates into the fact that each of the
matrices $x_i$ are of the form
$$
 \left(
\begin{array}{cccc}
 * & * & * & *\\
 0  & * & * & * \\
  \cdots& \cdots &  \cdots & \cdots\\
     0    &     0    &  0 & *
\end{array}
\right)
$$
where the size of the blocks $*$ is determined by the nonzero integers in the list 
$r_i-r_j$ with $i>j$. At any rate, one sees that $V$ has a submodule.

When $V$ is not simple, then $V$ has a submodule $W$. The subgroup of $\Gln$
preserving the flag $W\subseteq V$ is a parabolic subgroup. One easily
reverses the above argument to construct a one parameter subgroup with non-positive 
$\mu$.
\end{proof}
We will call a subspace $U$ of a vector space $V$ equipped with
a nondegenerate symmetric or alternating quadratic form $B$
\emph{totally isotropic with respect to $B$} if the restriction of $B$ to $U$ is zero.

\begin{proposition}
\label{p:stableiso}
 We fix a nondegenerate symmetric or alternating bilinear form $B$ on $V$
and let $G$ be the corresponding automorphism group preserving the 
form.
A point $x\in Z(\lie(G))$ is in the stable locus if and only if the
corresponding $R\otimes \Lambda$-module structure on $R\otimes V$ has no nontrivial submodules that
are totally isotropic with respect to  $B$.
\end{proposition}

\begin{proof}
 One uses the proof of the previous proposition and keeps track of
what it means for the one parameter subgroup to be inside $G$. We give more details 
for the case of $\Sptn$ and leave the case of $\Sotn$ to the reader.

Given a one parameter subgroup
$$
\lambda(t) = \left(
\begin{array}{cccc}
 t^{r_1} & & &\\
         & t^{r_2} &  & \\
         &         &  \ddots & \\
         &         &   & t^{r_{2n}}
\end{array}
\right)
$$
with image inside $\Sptn$ (using the standard symplectic form)
by swapping basis vectors we may arrange weights so that
$r_1>r_2>\ldots>r_n$ and $r_i = - r_{i+n}$. Suppose that we have a point
$x=(x_1,x_2,\ldots,x_g)\in Z(\lie(\Sptn))$ such that 
$\mu(x,\lambda)\le 0$. Then if we write $x_i$ as $n\times n$ blocks, that is
$$
x_i = \left(
\begin{array}{cc}
 A_i & B_i\\
 C_i & D_i
\end{array}
\right)
$$
we must have $C_i=0$. Hence there is an totally isotropic submodule.
The argument is easily reversed using the analogue of Witt's extension
theorem for alternating bilinear forms~\cite[p. 391]{jacobson:85}. More explicitly, we use the fact that one can extend an totally isotropic subspace to a maximal
totally isotropic subspace and all such maximal totally isotropic subspaces are the same
up to automorphism.
\end{proof}

\begin{proposition}
  Suppose that $G=\Gln,\ \Sptn$ or  $\Sotn$.
We have a $G^{ad}$-principal bundle
$$
Z(\lie(G))^s \rightarrow Z(\lie(G))^s/G^{ad}.
$$
\end{proposition}

\begin{proof}
 We need to show that the orbit map 
$$
\Psi : Z(\lie(G))^s \times G \rightarrow Z(\lie(G))^s \times Z(\lie(G))^s
$$
is a closed immersion. The action on the stable locus is proper by \cite[page 55, Corollary 2.5]{git:94}.
We need to show that points have trivial stabilisers. For $\Gln$ this follows
from \ref{p:simple} . Suppose that $\alpha\in G$ fixes a point $x$. This amounts
to the fact that $\alpha$ induces a $\Lambda$-module automorphism of $V$.

Let $\lambda\ne 0$ be an eigenvalue of $\alpha$ and consider
the $\lambda$-eigenspace of $\alpha$, $K=E_{\lambda}(\alpha)$.
As $\alpha$ preserves  the form $B$ we have that $K$ is totally isotropic. Explicitly if 
$v,w\in K$ then
\begin{eqnarray*}
 B(v,w) &=& B(\alpha(v), \alpha(w)) \\
 &=& \lambda^2 B(v,w).
\end{eqnarray*}
So $B(v,w) = 0 $ and $K$ is totally isotropic.
 Hence $K=V$.
\end{proof}

\subsection{Nori's Obstruction to Universal Bundles}
\label{ss:nori}

\begin{proposition}
 There exists a split Azumaya algebra with an action of
$\text{PGL}_n$ on $Z(\lie(\Gln))$.
\end{proposition}

\begin{proof}
 First let us write down the action of $\text{PGL}_n$ on the functor of points of
$Z(\lie(\Gln))$.
 Recall that an $R$-point of $Z(\lie(\Gln))$ is the same thing as a $R\otimes\Lambda$-module
structure on $R^n$. An $R$-point of $\Gln$ acts on $R^n$ giving a different module 
structure and hence a different map to $Z(\lie(\Gln))$. The centre acts trivially hence we
really have a $\text{PGL}_n$ action. This action lifts in a natural way to the Azumaya algebra 
$\text{End}_R(R^n)$.
\end{proof}

\begin{corollary}
 Let $U\subseteq Z(\lie(G))$ be a $G^{ad}$-stable open
subset on which $G^{ad}$ acts freely. Then there exists
an Azumaya algebra on $U/G$.
\end{corollary}

\begin{proof}
This follows by restricting the algebra in the proposition to $U$.
\end{proof}

We will denote the Azumaya algebra constructed above
by the symbol $\A(\rho, U)$. We will often abuse
notation and just write $\A$.

We denote the completion of the local ring at $0$ of 
$Z(\lie(G))/\!\!/G$ by $\widehat{\struct{Z/\!\!/G,0}}$ and its function
field by $\hat{k}_{Z/\!\!/G,0}$.  We define $\struct{Z/\!\!/G,0}^s$ by the 
Cartesian square
$$
\xymatrix{
 \spec(\struct{Z/\!\!/G,0}^s) \ar[r] \ar[d] & Z^s/G \ar[d] \\
 \spec(\struct{Z/\!\!/G,0}) \ar[r] & Z/\!\!/G.
}
$$
The ring $\widehat{\struct{Z/\!\!/G,0}^s}$ is defined by a similar Cartesian square. 

The class of $\A$ in $\Br(\widehat{\struct{Z/\!\!/G,0}^s})$ is called \emph{Nori's obstruction.}


The key result is:

\begin{theorem}
\label{t:obstruction}
 Denote by $\moduli{G}^0$ the connected component of the moduli
space containing the trivial bundle. If a universal bundle exists
on some open subset of $\moduli{G}^0$ then the Brauer class of 
$\A$ in $\text{Br}(\hat{k}_{Z/\!\!/G,0})$ vanishes.
\end{theorem}

\begin{proof}
Using \cite[Lemma 5.1]{drezet:04} and \ref{p:normal}, there is
an open neighbourhood $V$ of the trivial bundle
in $\N$ and a $G$ invariant morphism
$$
\phi : V\rightarrow {\rm T}_0(\N),
$$
where $0$ is the trivial bundle. By \cite[Proposition 5.1]{schmitt:02}, a universal bundle
produces a rational section
$$
\sigma:\moduli{G}^0\dashrightarrow \N.
$$
Composing $\phi\circ\sigma$ with the projection to the normal space to the orbit
$$
{\rm T}_0(\N) \rightarrow {\rm N}_{\text{orb(0)}}
$$
and recalling that the slice theorem tells us that
$$
\widehat{\struct{M,0}} \cong \widehat{\struct{Z/\!\!/G,0}}
$$
we see that the map
$$
Z^s\rightarrow Z^s/G^{\ad}
$$
has a section over $\hat{k}_{Z/\!\!/G,0}$. Hence our algebra is split.
\end{proof}

\section{Some Remarks on Generic Quaternion Algebras}

We will recall here some facts on the generic quaternion
algebra over the field $k(x,y)$ that will be needed below.
A detailed exposition can be found in \cite[Sections 10 and 11]{draxl:83}.
The generic quaternion algebra, denoted $\Ba_{k(x,y)}$,
is the algebra generated by the symbols $a$ and $b$ subject to the
relations $a^2=x$, $b^2=y$ and $ab=-ba$. It is split by the 
quadratic extension $k(\rx,y)$. 

\begin{proposition}
\label{p:quaternion}
\begin{enumerate}
 \item 
 One can realise $\Ba_{k(x,y)}$ as the subalgebra of 
$M_2(k(\rx,y))$ generated by the matrices
$$
\left(\begin{array}{cc}
      \rx & 0 \\
       0 & -\rx
     \end{array}\right)\quad \text{and}\quad
\left(\begin{array}{cc}
      0 & 1 \\
       y & 0
     \end{array}\right)
$$

\item Let $\sigma$ be the generator of 
${\rm Gal}(k(\rx,y)/k(x,y))$. Extend the action of ${\rm Gal}(k(\rx,y)/k(x,y))$
to $M_2(k(\rx,y))$ by defining
$$
\left(\begin{array}{cc}
      \alpha & \beta \\
       \gamma &  \delta
     \end{array}\right)^\sigma =
\left(\begin{array}{cc}
      \delta^\sigma & \frac{\gamma^\sigma}{y} \\
       \beta^\sigma y &  \alpha^\sigma
     \end{array}\right).
$$
Then $\Ba_{k(x,y)}$ is the subalgebra fixed by ${\rm Gal}(k(\rx,y)/k(x,y))$.
\end{enumerate}
\end{proposition}

\begin{proof}
 The first part is carried out in \cite[section 10]{draxl:83} and the
second part follows from the calculations in \cite[pages 77--78]{draxl:83}
by some linear algebra.
\end{proof}

\section{The Case of $\Sptn$}

Consider the field $\hat{L}=k((x,y))$ and the generic quaternion algebra $\Ba\otimes\hat{L}=\hat{\Ba}$
over $\hat{L}$.  It is split by the
field extension $\hat{K}=k((\sqrt{x},y))$. We will construct
an $\hat{L}$-point of $\spec(\widehat{\struct{Z/\!\!/G,0}^s})$ that pulls back Nori's
obstruction to $\hat{\Ba}\otimes M_n(\hat{L})$. (Recall the definition of 
$\spec(\widehat{\struct{Z/\!\!/G,0}^s})$ from \S 4.2.)
 By \cite[IV Corollary 2.6]{milne:80} and \ref{t:obstruction}, 
this will show that no universal bundle exists on the moduli space for 
$G=\Sptn$.

\subsection{Construction of the $\hat{L}$-point} 

Let $R=k[\sqrt{x},y]$ and set 
$$
A=\left(\begin{array}{cc}
 \sqrt{x}D & 0 \\
 0 & -\sqrt{x}D
\end{array}
 \right)\quad
C=\left(\begin{array}{cc}
 0 & I_n \\
 yI_n & 0
\end{array}
 \right).
$$
where $D=\diag(1,2,\dots,n)$.
 We consider $R^{2n}$ with the symplectic form $B$
determined by the matrix
$$\left[\begin{array}{cc}0&-I_n\\I_n&0\end{array}\right]$$
and consider $A,C\in \End_R(R^{2n})$. 
There is an anti-involution $r$ of $M_{2n}(R)$ determined by taking 
adjoints that is 
$$
B(r(f)(v),w) = B(v,f(w)).
$$
One easily checks that 
$r(A)=-A$ and $r(C)=-C$. 

Let $V$ be a vector space of dimension $2n$ over $k$ with the standard symplectic form $B$. Recall that the symplectic group $\Sptn$ is defined as $\{g \in GL(V) | B(v,w)=B(gv,gw) \text{ for all } v,w \in V\}$. The centre of $\Sptn$ is $\{\pm I_{2n}\}$. Taking the quotient gives us the projective symplectic group $\PSptn$. We also have the group of symplectic similitudes $\GSptn$ defined as $\{g \in GL(V) | B(v,w)=\lambda B(gv,gw) \text{ for all } v,w \in V, \lambda \in k^*\}$. The scalar matrices form a subgroup of $\GSptn$ and the quotient is isomorphic to $\PSptn$. 

Consider the $R\otimes \Lambda$-module structure on $R^{2n}$ where
the action of $x_i$ is defined by
$$
x_i = \left\{ \begin{array}{cc}
        A & i=1 \\
 	yC & i\ge 2 .

      \end{array} \right. ,
$$
Note that $x_i\in \lie(\Sptn)$.
In view of the discussion at the start of section \ref{s:azu}, this yields
an $R$-point 
$$\phi : \spec(R) \rightarrow Z(\lie(\Sptn)).$$
There is an associated $K=k(\sqrt{x},y)$-point 
$$\phi_0:\spec(K)\rightarrow Z(\lie(\Sptn)).$$

There is an action of $\Z/2\Z$ on $R$ sending $\sqrt{x}$ to $-\sqrt{x}$ and fixing $y$. Viewing
$\text{PSp}_{2n}$ as a quotient of $\text{GSp}_{2n}$ we define an inclusion
$$
\Z/2\Z\hookrightarrow \text{PSp}_{2n}
$$
with image $C$. Let $R_y$ be the localisation of $R$ obtained 
by inverting $y$. We have an $R_y$ point of $Z(\lie(\Sptn))$ denoted by $\phi_y$which is  
obtained from $\phi$.
With this definition $\phi_y$ is equivariant for the
$\Z/2\Z$-action as
$$
CAC^{-1}=-A
$$
and $C$ acts trivially by conjugation on the other $x_i$ defined above.

 Setting $S=k[x,y]_y$, we obtain a map
$$
\bar{\phi_y}:\spec(S)\rightarrow Z(\lie(\Sptn))/\!\!/\Sptn
$$
by passing to quotients.

\begin{proposition}
 \begin{enumerate}
  \item The map $\bar{\phi_y}$ induces a morphism
 $$\bar{\phi}_0:\spec(k(x,y))\rightarrow Z(\lie(\Sptn))^s/\PSptn.$$
  \item Abusing notation slightly, we have $\bar{\phi_0}^*(\A) = \Ba_{k(x,y)}\otimes M_n(k(x,y))$.
 \end{enumerate}
\end{proposition}

\begin{proof}

\noindent (1) Set $K=k(\sqrt{x},y)$. We need to show that the $K\otimes \Lambda$-module structure defined above
on $K^{2n}$ has no totally isotropic submodules.
Suppose that there exists an totally isotropic submodule $V$. Then $V$ must contain an
 $A$ eigenvector
$v$ as the endomorphism $A$ is semisimple with distinct eigenvalues. But then one
can show that $B(v,C(v))\ne 0$, so that the $K\otimes \Lambda$
submodule generated by
$v$  is not totally isotropic.
 
\noindent (2)  In \ref{p:quaternion} part 2, the generic quaternion algebra was constructed
as a geometric quotient of the matrix algebra $M_2(K)$ with a prescribed $\Z/2\Z$-action.

There is a $\Z/2\Z$-action on $M_2(K)\otimes M_n(K)$ where $\Z/2\Z$ acts on the second factor
via its Galois action on $K$.
Note that 
$$
C\left(
\begin{array}{cc}
 X & Y \\
 W & Z
\end{array}\right) C^{-1}
=\left(
\begin{array}{cc}
 Z & y^{-1}W \\
 yY & X
\end{array}\right)
$$
 The discussion above shows that this is the 
pullback of an action of $\Z/2\Z$ on $Z(\lie(\Sptn))$ via $\phi_0$. The result follows via taking quotients.  
\end{proof}

\begin{corollary}
 There exists a map
$$
\hat{\phi}:\spec(k((x,y)))\rightarrow \spec(\widehat{\struct{Z/\!\!/G,0}^s}).
$$
such that $\hat{\phi}^*\A = \Ba_{k(x,y)}\otimes M_n(k(x,y))$.
\end{corollary}

\begin{proof}
 The map $\phi$ sends $0$ to $0$ hence there is a map on completions of local
rings. Therefore we obtain a map 
$$
\phi_c:\spec(k((\sqrt{x},y)))\rightarrow \widehat{\struct{Z,0}}.
$$
In view of the above proposition
it suffices to show that it descends to a morphism 
$$
\spec(k((x,y)))\rightarrow \widehat{\struct{Z/\!\!/G,0}}.
$$
The fact that $\phi$ sends 0 to 0 gives a commutative diagram
of local rings and local homomorphisms
$$
\xymatrix{
 k((\sqrt{x},y)) & \struct{Z,0} \ar[l] \\
 k((x,y)) \ar[u] & \struct{Z/\!\!/G,0} \ar[u] \ar[l].
}
$$
with $\phi_c$ induced by completion from the top line.
The image of the maximal ideal of $\struct{Z/\!\!/G,0}$ must be contained in the
$\struct{Z/\!\!/G,0}$ submodule of $k((x,y))$ generated by $(x,y)$. But the field
$k((x,y))$ is complete with respect to the induced topology and hence
we obtain our map.
\end{proof}

\begin{corollary}
 There is no universal bundle over the generic point
of $\moduli{\Sptn}$.
\end{corollary}

\begin{proof}
 The algebra $\B$ is not split by \cite[Corollary 4, page 82]{draxl:83}.
It follows the algebra $\A$ is not split over $\widehat{\struct{Z/\!\!/G,0}^s}$. The
ring $\widehat{\struct{Z/\!\!/G,0}^s}$ is regular as the map
$$
Z^s\rightarrow Z^s/\PSptn
$$
is a principal bundle and hence $Z^s/\PSptn$ is smooth. By \cite[IV Corollary 2.6]{milne:80}, the algebra does not
split over the generic point of $\widehat{\struct{Z/\!\!/G,0}^s}$. The required result follows
from \ref{t:obstruction}.
\end{proof}

\section{The Case of $\Sotn$}

Set $R=k[\rx,y]$. We consider on $R^{2n}$ the symmetric form given 
by the matrix
$$
\left(\begin{array}{cc}
       0 & I_n \\
	I_n & 0
      \end{array}
\right).
$$
We set $\Sotn$ to be the orthogonal group preserving the associated symmetric form $B$.
 We define an $R\otimes \Lambda$-module structure on $R^{2n}$ where the 
action is given by the following matrices :
$$
x_1 = 
\left(\begin{array}{cc}
       \rx D &0 \\
	 0& -\rx D
      \end{array}
\right)
$$
and for $i>1$
$$
x_i = 
\left(\begin{array}{cc}
       \rx A &  yZ\\
	y^2 Z &  -\rx A
      \end{array}
\right).
$$
In the above
$D=\diag(1,2,\dots,n)$, $A=(a_{ij})$, and  $Z=(z_{ij})$ are given by 
$$
a_{ij} = \left\{ 
	\begin{array}{cc}
	 1 & \text{if } i=1 \text{ or } j= 1\\
	 0 & \text{otherwise} 
	\end{array}
\right. \quad
z_{ij} = \left\{ 
	\begin{array}{cc}
	 -1 & \text{if } j=1,i\ne 1\\
	 1 & \text{if } i=1, j\ne 1 \\
	 0 & \text{otherwise}.
	\end{array}
\right.
$$
Note that $x_i\in \lie(\Sotn)$.
This gives an $R$-point $\phi:\spec(R)\rightarrow Z(\lie(\Sotn))$
as described at the start of section \ref{s:azu}. The definition seems a little
arbitrary but it is this way to ensure that it is easy to 
show that our point hits the stable locus of $Z(\lie(\Sotn))$.
 
Denote by $\PSotn$ the group $\Sotn/ Z(\Sotn)$, that is the adjoint form of $\Sotn$.
 We denote by $\mathrm{GSO}_{2n}$ the group of linear transformations that preserve the standard bilinear form up to a scalar. Finally we define $\mathrm{PGSO}_{2n}$ to be
$$
\mathrm{GSO}_{2n} / Z(\mathrm{GSO}_{2n}).
$$

There is a natural identification $\PSotn \cong \mathrm{PGSO}_{2n}$ coming from the
fact that scalar matrices are central.

This identification  gives a map
$$
\psi:\Z/2\Z \rightarrow \PSotn
$$
with image 
$$
\left(\begin{array}{cc}
       0 &  I_n\\
	yI_n &  0
      \end{array}
\right) \stackrel{\text{defn}}{=} C
$$
Its scheme theoretic image is a closed subgroup $H<\PSotn$
defined over $k(y)$.

Denote by $\sigma$ the automorphism of $R$ fixing $y$ and sending 
$\rx$ to $-\rx$. Note that
$$
C x_i C^{-1} = x_i^\sigma.
$$
Hence if we set $R_y= k[\rx,y]_y$ we obtain an induced $R_y$ point that
induces a map
$$
\spec(k[x,y]_y)\rightarrow Z(\lie(\Sotn))/\!\!/\Sotn.
$$

\begin{proposition}
 \begin{enumerate}
  \item There is an induced morphism 
$$
\bar{\phi}_0:\spec(k(x,y))\rightarrow Z(\lie(\Sotn))^s/\PSotn.
$$
\item $\bar{\phi}_{0}^*\A = \Ba_{k(x,y)}\otimes M_n(k(x,y)).$
 \end{enumerate}
\end{proposition}

\begin{proof}
 \noindent (1) We need to show that our map $\phi$ sends the generic point
to the stable locus. By \ref{p:stableiso} this amounts to showing that there are no totally isotropic submodule.
Suppose that we have an totally isotropic submodule $V\subseteq k(x,y)^{2n}$. As $x_1$ is semisimple with
distinct eigenvalues, $V$ must contain an $x_1$ eigenvector. We may as well assume that is $e_i\in k(x,y)^{2n}$ where $[e_i]_j=\delta_{ij}$.
But then $B(x_2(e_i),x_2(e_i))\ne 0$ so that the $K\otimes \Lambda$
module generated by $e_i$ is not totally isotropic. 

\noindent (2) This is similar to the analogous result for the
symplectic group. The algebra $\A$ is a quotient of the trivial 
Azumaya algebra of rank $4n^2$ over $Z$ where $\Sotn$ acts on this algebra
by conjugation. Pulling this action back via $\psi$ and $\phi$ one obtains the
action described in \ref{p:quaternion}.
\end{proof}

\begin{corollary}
 There exists a map
$$
\hat{\phi}:\spec(k((x,y)))\rightarrow \spec(\widehat{\struct{Z/\!\!/G,0}^s}).
$$
such that $\hat{\phi}^*\A = \Ba_{k(x,y)}\otimes M_n(k(x,y))$.
\end{corollary}

\begin{proof}
 As for the case of the symplectic group.
\end{proof}

\begin{corollary}
 There is no universal bundle on the generic point of $\moduli{\Sotn}^0$.
\end{corollary}

\begin{proof}
 This follows from the above discussion and \ref{t:obstruction}.
\end{proof}

\bibliographystyle{alpha}
\bibliography{./../bibliography/alggrp.bib,./../bibliography/mybib.bib,./../bibliography/sga.bib,./../bibliography/mybib2.bib}
\end{document}

%% file: mydefs.tex
\newcommand{\Z}{\ensuremath{{\mathbb Z}}}
\newcommand{\Q}{\ensuremath{\mathbb Q}}

\newcommand{\coh}[2]{{\rm H}^{#1}(#2)}

\newcommand{\shom}[2]{{\mathcal Hom}_{#1}(#2)}

\newcommand{\struct}[1]{{\mathcal O}_{#1}}

\newcommand{\spec}{{\rm Spec}}

\newcommand{\gm}{{\mathbb G}_m}